\documentclass[12pt]{amsart}

\usepackage{enumerate}
\usepackage{mathrsfs}	

\usepackage{tikz-cd}

\usepackage[
	pdftex,
    pdfauthor={Joshua Mundinger},
    pdftitle={Twisting the Infinitesimal Site},
    pdfkeywords={twist, differential operators},
]{hyperref}

\usepackage[
	backend=bibtex,
	style=alphabetic,
	citestyle=alphabetic,
	maxnames=100, 
    doi=false,
	giveninits=true,
]{biblatex}
\DeclareNameAlias{default}{family-given}

\usepackage{amsthm}

\theoremstyle{plain}
\newtheorem{theorem}{Theorem}[section]
\newtheorem{lemma}[theorem]{Lemma}
\newtheorem{proposition}[theorem]{Proposition}
\newtheorem{corollary}[theorem]{Corollary}

\theoremstyle{definition}
\newtheorem{definition}[theorem]{Definition}
\newtheorem{example}[theorem]{Example}

\newtheorem{remark}[theorem]{Remark}

\usepackage{amsmath,amssymb}

\newcommand{\OO}{\mathcal O}

\DeclareMathOperator{\colim}{\varinjlim}
\DeclareMathOperator{\ch}{char}
\DeclareMathOperator{\End}{End}
\DeclareMathOperator{\gr}{gr}
\DeclareMathOperator{\Hom}{Hom}
\DeclareMathOperator{\Maps}{Maps}

\DeclareMathOperator{\rank}{rank}
\DeclareMathOperator{\Spec}{Spec}
\DeclareMathOperator{\Sym}{Sym}

\DeclareMathOperator{\cone}{cone}
\DeclareMathOperator{\fib}{fiber}
\DeclareMathOperator{\Fun}{Fun}

\DeclareMathOperator{\Br}{Br}
\DeclareMathOperator{\Pic}{Pic}
\DeclareMathOperator{\NS}{NS}

\newcommand{\pop}{ {[p]} }

\newcommand{\Diff}{\mathrm{Diff}}
\newcommand{\Dtilde}{\widetilde{\mathcal D}}

\newcommand{\RGet}{R\Gamma_{\acute{e}t}}

\title{
	Twisting the Infinitesimal Site
}
\author{Joshua Mundinger}
\address{University of Wisconsin-Madison\\ Madison\\ WI, USA}
\email{jmundinger@wisc.edu}

\date{August 23, 2024}


\begin{document}

\begin{abstract}
	We classify twistings of Grothendieck's differential operators on a smooth variety $X$ in prime characteristic $p$. We prove isomorphism classes of twistings are in bijection with $H^2(X,\mathbb{Z}_p(1))$, the degree 2, weight 1 syntomic cohomology of $X$. We also discuss the relationship between twistings of crystalline and Grothendieck differential operators. 
	Twistings in mixed characteristic are also analyzed.
\end{abstract}

\maketitle
\vspace{-1em}
\section{Introduction}

Beilinson and Bernstein defined and studied rings of twisted differential operators in connection with their study of representations of semisimple Lie algebras \cite{bb93}.
If $\Diff_X$ is the ring of differential operators on a smooth space $X$, then a ring of twisted differential operators on $X$ is essentially a filtered ring $D = \cup_i D_{\leq i}$ such that 
\[\gr D \cong \gr \Diff_X\] 
as Poisson algebras. If $X$ is a holomorphic manifold or complex algebraic variety, then such rings $D$ are generated in degree 1 and are thus classified by the extension 
\[ 0 \to \OO_X \to D_{\leq 1} \to T_X \to 0\]
of sheaves of Lie algebras. The set of isomorphism classes of such extensions is thus an abelian group isomorphic to $H^2(X,\Omega_X^{\geq 1})$, which coincides with the Hodge filtration subspace $F^1H_{dR}^2(X)$ when $X$ is a compact K\"ahler manifold or a proper complex algebraic variety \cite[Lemma 2.1.6]{bb93}.

In this paper, we study twisted differential operators when $X \to S$ is a smooth morphism of schemes over a nonzero characteristic base $S$. 
When $S$ is not of characteristic zero, then there are multiple inequivalent notions of differential operators on $X$. Two main rings of interest are the \emph{crystalline differential operators} $\mathcal D^{crys}_X$ and the \emph{Grothendieck differential operators} $\Diff_X$. The ring of crystalline differential operators $\mathcal D^{crys}_X$ is the enveloping algebra of the tangent sheaf $T_{X/S}$, with associated graded $\Sym_{\OO_X} T_{X/S}$; the ring of Grothendieck differential operators, or ``full'' ring of differential operators, $\Diff_X$ includes divided powers of partial derivatives, and has associated graded $\Gamma_{\OO_X} T_{X/S}$, the divided symmetric power algebra. There is a filtered map 
\[ \mathcal D^{crys}_X \to \Diff_X,\]
inducing the canonical map $\Sym_{\OO_X} T_{X/S} \to \Gamma_{\OO_X} T_{X/S}$ on associated graded.
In positive characteristic, the map is neither injective nor surjective.

We study twistings $D$ of Grothendieck's differential operators $\Diff_X$. 
These algebras $D$ are not generated in degree 1 and so are not determined by the extension class of $D_{\leq 1}$. Instead, such algebras are classified by a $p$-adic analogue of $F^1 H^2_{dR}(X)$.
Given a prime $p$, the \emph{weight 1 syntomic cohomology} of a scheme $X$ is by definition
\[ R\Gamma(X, \mathbb Z_p(1)) = \RGet(X,\mathbb G_m)^\wedge_p[-1],\]
the derived $p$-completion of $\RGet(X,\mathbb G_m)[-1]$.
\begin{theorem} \label{theorem: main}
	Let $k$ be a perfect field of characteristic $p$ and $X \to \Spec k$ a smooth variety. Then isomorphism classes of rings of twisted  Grothendieck differential operators are in bijection with 
	\[ H^2(X, \mathbb Z_p(1)),\]
	the degree 2, weight 1 syntomic cohomology of $X$.
\end{theorem}

The group $H^2(X,\mathbb Z_p(1))$ can be expressed in terms of the Picard and Brauer groups via the short exact sequence 
\[\begin{tikzcd}
	0 & {\Pic(X)^\wedge_p} & {H^2(X,\mathbb Z_p(1))} & {T_p\Br(X)} & 0
	\arrow[from=1-1, to=1-2]
	\arrow[from=1-2, to=1-3]
	\arrow[from=1-3, to=1-4]
	\arrow[from=1-4, to=1-5]
\end{tikzcd},\]
where $T_p(-)$ is the $p$-adic Tate module and $\Pic(X)^\wedge_p$ is the classical $p$-adic completion of the Picard group.
The first map is induced by the Chern class $c_1: \Pic(X) \to H^2(X,\mathbb Z_p(1))$ \cite[§7]{bhattlurie22}. In terms of twisted differential operators, this assignment sends a line bundle $\mathcal L$ to $\Diff_{\mathcal L}$, the ring of differential operators on $\mathcal L$.
If $k$ is algebraically closed, then the syntomic cohomology group $H^2(X,\mathbb Z_p(1))$ may be computed via crystalline cohomology as 
\[ H^2(X, \mathbb Z_p(1)) \cong H^2(X, W\Omega^{\geq 1}_X)^{\varphi= p},\]
see §3.2.1 below. This isomorphism transforms the syntomic Chern class to the crystalline Chern class,
of interest in the Tate conjecture.

If $X$ is a variety over a field $k$, then $H^2(X,\Omega_X^{\geq 1})$ is a $k$-vector space, while $H^2(X,\mathbb Z_p(1))$ is a $\mathbb Z_p$-module regardless of $k$. Here is the intuition why twisted differential operators should be controlled by a $\mathbb Z_p$-module and not a vector space. Fix an invertible function $f$; if $f^\lambda$ may be differentiated, then conjugation by $f^\lambda$ defines an automorphism of $\Diff_X$. The function $f^\lambda$ satisfies the differential equation $(d - \lambda d\log(f))f^\lambda = 0$, which makes sense for any parameter $\lambda$ in the base field. However, over a positive characteristic field $k$, we must specify not only the first derivative but all divided power derivatives. If $\partial$ is a derivation admitting divided powers $\{\partial^{(i)}\}_{i \geq 1}$, then Faà di Bruno's formula gives 
\[ 
\partial^{(i)}(f^\lambda) = 
\sum \binom{m}{m_1,\ldots, m_i} \binom{\lambda}{m} f^{\lambda - m} \prod_{j=1}^i \left( \partial^{(j)} f\right)^{m_j},
\]
where the sum runs over all $i$-tuples $(m_1,\ldots, m_i)$ such that $\sum_j j m_j = i$, while $m = \sum_j m_j$.
If we wish to form $\partial^{(i)}f^\lambda$ for all $i \geq 0$, then $\binom{\lambda}{m}$ must make sense for all $m \geq 0$, that is, the parameter $\lambda$ must be a $k$-point of a free binomial ring on one generator. The set of residue characteristic $p$ points of the free binomial ring on one generator is exactly $\mathbb Z_p$ \cite{elliott06}. The same considerations explain why T.\ Bitoun's $b$-function in positive characteristic is a locally constant function on $\mathbb Z_p$ \cite{bitoun18}.

One might ask whether the map $\mathcal D^{crys}_X \to \Diff_X$ can be twisted. Given an algebra of twisted Grothendieck differential operators $D$ on smooth $X \to \Spec k$, we construct in §\ref{subsection: compatibility}
a Frobenius-split twisting $\Dtilde^{crys}$ of $\mathcal D^{crys}_X$ depending on $D$ and a map $\Dtilde^{crys} \to D$ such that the kernel is generated by the ideal of the zero section of $T^*X = \Spec Z(\Dtilde^{crys})$.

We also study a mixed characteristic version of this problem, when $S = \Spec W_m(k)$ for finite $m$. If $X \to \Spec W_m(k)$ is smooth and of finite type, then in Theorem \ref{theorem: mixed-characteristic-syntomic-comparison} we construct a map from $H^2(X,\mathbb Z_p(1))$ to isomorphism classes of twisted differential operators on $X$. This map may or may not be an isomorphism, depending on the action of Frobenius on $R\Gamma(X_0,\OO_{X_0})$, where $X_0$ is the special fiber of $X$.

We caution the reader that our twisted differential operators are in the sense of Beilinson-Bernstein, and do not involve $q$-difference operators as in Gros-Le Stum-Quirós \cite{groslestumquiros22}.

\subsection*{Acknowledgments} 
The author thanks Victor Ginzburg for his encouragement of this project.
The author is grateful to
Daniel Bragg,
Luc Illusie,
Akhil Mathew,
and
Vadim Vologodsky 
for useful conversations and communications,
and to the anonymous referees for helpful suggestions.
Ben Antieau hosted stimulating workshops at Northwestern University which considerably advanced the progress of this work. The appendix was written through discussions with Dima Arinkin.
The author was partially supported by NSF Graduate Research Fellowship DGE 1746045.

\section{Twistings and twisted differential operators}

\subsection{Algebras of twisted differential operators}
\label{subsection: tdo}

Let $S$ be an affine scheme and $X \to S$ a morphism of schemes. Undecorated products of schemes over $S$ are understood to be products over $S$.

\begin{definition}\cite[§1.1.4]{bb93}
	A differential algebra on $X/S$ is a sheaf of associative algebras $D$ on $X$ equipped with a morphism $\iota: \OO_X \to D$ such that the image of $\OO_S$ is central and such that as an $\OO_X$-bimodule, $D$ is supported on the diagonal of $X \times X$.
\end{definition}

A differential algebra on $X/S$ is equipped with a canonical filtration $D_{\leq -1 } = 0, D_{\leq i} = \{Q \in D \mid [Q,f] \in D_{\leq i-1} \text{ for all }f \in \OO_X\}$.
Since $D$ is supported on the diagonal in $X \times X$, this filtration is complete.

The first example of a differential algebra is Grothendieck's ring of differential operators $\Diff_X$, defined as follows: 
$\Diff_{X,\leq i}$ consists of $\OO_S$-linear operators $Q: \OO_X \to \OO_X$ such that for all $f_0,f_1,\ldots, f_i \in \OO_X$, the iterated commutator $[\ldots[[Q,f_0],f_1],\ldots,f_i]$ vanishes.
If $\OO_S$ contains $\mathbb Q$, then $\Diff_X$ is generated by operators of degree at most 1; if $\OO_S$ contains $\mathbb F_p$, then $\Diff_X$ is not Noetherian, and contains operators such as $(p^r)!^{-1} (d/dt)^{p^r}$. If $X/S$ is smooth, then the associated graded ring $\gr \Diff_X$ is isomorphic to the divided power symmetric algebra $\Gamma_{\OO_X} T_{X/S}$ on the tangent sheaf $T_{X/S}$ \cite[2.6 Proposition]{berthelotcrystalline}.

We are interested in twisted versions of Grothendieck's differential operators $\Diff_X$. Definition \ref{defn: tdo} below modifies Beilinson and Bernstein's definition of twisted differential operators \cite[§2]{bb93} to deal with divided powers.
If $D$ is a differential algebra on $X/S$ and $f \in \OO_X$, then the commutator with $f$ sends $D_{\leq \ast}$ into $D_{\leq \ast - 1}$ and thus defines $[-,f] :\gr_\ast D \to \gr_{\ast - 1} D$.
The assignment $f \mapsto [-,f]$ satisfies the Leibniz rule, and thus defines a bilinear map $\Omega^1_{X/S} \times \gr_\ast D \to \gr_{\ast - 1} D$. If $f_1\in \OO_X$ and $f_2 \in \OO_X$, then $f_1$ and $f_2$ commute, so $[-,f_1]$ and $[-,f_2]$ are commuting operators on $\gr D$.
Thus we obtain a pairing
\[ \Sym^i \Omega^1_{X/S} \times \gr_\ast D \to \gr_{\ast - i} D,\]
a generalization of the principal symbol of a differential operator.
\begin{definition}\label{defn: tdo}
	If $X/S$ is smooth, then an \emph{algebra of twisted Grothendieck differential operators} on $X/S$, or a \emph{Grothendieck tdo} on $X/S$, is a differential algebra $D$ such that 
	$\iota: \OO_X \to D_{\leq 0}$ is an isomorphism and the bilinear map 
	\[ 
		\Sym^i \Omega^1_{X/S} \times \gr_i D \to \gr_0 D \cong \OO_X
	\]
	defined by 
	\[ \langle df_1\cdots df_i, Q \rangle =[[[Q,f_1],\ldots ],f_i]
	\]
	is a perfect pairing.
\end{definition}

\begin{remark}
	For the rest of this paper, we will refer to Grothendieck tdo's simply as tdo's.
\end{remark}

Now assume $X/S$ is smooth. If $D$ is a tdo, then $\gr D$ is the divided power symmetric algebra on $T_{X/S}$. $\Diff_X$ is a tdo. The category of tdo's on $X$ is a groupoid, since a morphism of differential algebras preserves $\iota$ and thus the canonical filtration. If $u: U \to X$ is a smooth morphism, then pullback of left $\OO$-modules sends differential algebras on $X$ to differential algebras on $U$ \cite[§1.5]{bb93}; if $U \to X$ is étale, then $u^* \Omega^1_{X/S} = \Omega^1_{U/S}$, so the pullback of a tdo is a tdo. Hence tdo's form a 1-stack $TDO(X/S)$ on $X_{\acute{e}t}$.

\begin{example}
	If $\mathcal L$ is a line bundle on $X$,
	then $\Diff_{\mathcal L}$ is a tdo on $X$.
\end{example}	

\subsection{The stacky approach}\label{subsection: stacky approach}
Gaitsgory and Rozenblyum developed a general theory of twistings in order to correctly formulate twisted $D$-modules on a stack in characteristic zero using the de Rham stack \cite{gaitsgoryrozenblyum14}. In nonzero characteristic, different flavors of $D$-modules correspond to different variants of the de Rham stack. Grothendieck's differential operators correspond to the infinitesimal site $(X/S)_{inf}$ associated to $X/S$, defined below. 
We now recall Gaitsgory and Rozenblyum's formulations of twists \cite[§6.6]{gaitsgoryrozenblyum14} and prove that twistings of the infinitesimal site are equivalent to the tdo's defined in §\ref{subsection: tdo}.

\begin{definition}
	Given a morphism $X\to S$ of schemes, the infinitesimal prestack $X_{inf}: \mathrm{Aff}_{/S} \to \mathrm{Set}$ is defined by 
	\[ X_{inf}(Z) = X(Z_{red}).\]
\end{definition}

There is a natural map $\rho: X \to X_{inf}$ over $S$.

\begin{definition}
	A \emph{twisting of the infinitesimal site} of $X$ is an étale $\mathbb G_m$-gerbe on $X_{inf}$ equipped with a trivialization of its pullback along $\rho: X \to X_{inf}$.
\end{definition}

The \v{C}ech nerve $X_\bullet$ of $X \to X_{inf}$ is the formal completion of $X^{\times \bullet + 1}$ along the diagonal $X \to X^{\times \bullet + 1}$. 

\begin{proposition}\cite[Lemma 1.2.4]{gaitsgoryrozenblyum14}
	\label{lemma: cech-nerve}
	If $X/S$ is formally smooth, then $X_{inf} \simeq |X_\bullet|$ in $\infty$-groupoid-valued prestacks over $S$.
\end{proposition}
The proposition means that we may regard the $\infty$-category of abelian sheaves on $X_{inf}$ as the $\infty$-category of abelian sheaves on $X$ equipped with descent data along $X_\bullet$.

\begin{proposition}\cite[§6.6.4]{gaitsgoryrozenblyum14}
	If $X/S$ is smooth, then the groupoid of algebras of twisted differential operators on $X$ is equivalent to the groupoid of twistings of $X \to X_{inf}$.
\end{proposition}
\begin{proof}
	By Lemma \ref{lemma: cech-nerve}, $|X_\bullet| \simeq X_{inf}$.
	An étale $\mathbb G_m$-gerbe on $X_\bullet$ with a trivialization of its pullback along $X \to X_\bullet$ has two different trivializations upon pullback to $\widehat{X \times X}$; hence we obtain a line bundle $D$ on $\widehat{X \times X}$.
	The pullback of $D$ along $X \to \widehat{X \times X}$ is trivialized by a bimodule homomorphism $\iota: \OO_X \cong D_{\leq 0} \subseteq D$.
	The face maps along the diagrams 
	\[X_3 \to X_2 \to X_0 = \widehat{X \times X}\]
	define an associative product on $D$; the degeneracy and face maps imply $\iota$ is a morphism of associative rings. Since $D$ is a line bundle on $\widehat{X \times X}$, the commutator pairings $\Sym^i\Omega^1_{X/S} \times \gr_i D \to \OO_X$ are isomorphisms. 
	Conversely, reversing the arguments above shows a tdo $D$ defines a line bundle on $X_0$ along with associated coherence data on $X_\bullet$, descending to a gerbe on $|X_\bullet|$.
\end{proof}

\begin{corollary}\cite[Corollary 6.5.3]{gaitsgoryrozenblyum14}
	\label{corollary: twistings-classified-by-fiber}
	The groupoid of twistings of the infinitesimal site is controlled by the complex of abelian groups
	\[ \tau^{\leq 2} \fib(\RGet(X_{inf},\mathbb G_m) \to \RGet(X,\mathbb G_m)).\]
\end{corollary}
\begin{proof}
	It is well-known that gerbes banded by an abelian group object $A$ are classified by $H^i(-,A)$ for $i \in \{0,1,2\}$, going back to work of Giraud \cite{Gir71}, and thus twistings are classified by a mapping fiber on cohomology. In our setting an ideal reference is not available, so we sketch the proof.

	Let $B^2\mathbb G_m$ be the étale classifying stack of the group stack $B\mathbb G_m$; then $B^2\mathbb G_m$ represents the 2-stack of étale $\mathbb G_m$-gerbes \cite[§6.1.1]{gaitsgoryrozenblyum14}.
	For a scheme $Y/S$, $\RGet(Y,\mathbb G_m)[1]$ is the complex corresponding under Dold-Kan to the  mapping spectrum $\mathcal{M}aps(Y,B\mathbb G_m)$, so
	the mapping space $\Maps(Y,B^2\mathbb G_m)$ is given by 
	\[ \Maps(Y, B^2\mathbb G_m) \simeq \tau^{\leq 0}(\RGet(Y,\mathbb G_m)[2]).\]
	Left Kan extending gives the same result if $Y/S$ is a prestack.
	Now given a map $Y \to Y'$, the space of $\mathbb G_m$-gerbes on $Y'$ equipped with a trivialization of its pullback to $Y$ is the homotopy fiber of 
	\[ \Maps(Y',B^2\mathbb G_m) \to \Maps(Y, B^2\mathbb G_m).\qedhere\]
\end{proof}

If $S$ is of characteristic zero, then Gaitsgory and Rozenblyum calculated the groupoid of twistings as follows: $\RGet(X_{inf},\mathbb G_m)$ is computed by the logarithmic de Rham complex $d\log: \OO_X^\times \to \Omega^{\geq 1}_X$. Hence the desired fiber is exactly $\RGet(X,\Omega_X^{\geq 1})$, recovering the classical answer of Beilinson and Bernstein \cite[Lemma 2.1.6]{bb93}.
To compute the groupoid of twistings in nonzero characteristic, we must compute $\RGet(X_{inf},\mathbb G_m)$.

\section{F-divided structures and syntomic cohomology}

Fix a perfect field $k$ of characteristic $p > 0$. It is a well-known theorem of Katz that if $X/\Spec k$ is smooth, then modules over $\Diff_X$ are equivalent to $F$-divided sheaves of $\OO_X$-modules \cite[Theorem 1.3]{gieseker75}. This is an extension of Cartier's classical theorem on descent along the Frobenius. Ogus later applied this idea to compute cohomology of coherent sheaves on the infinitesimal site in terms of crystalline cohomology \cite{ogus75}. 
Berthelot extended Katz's theorem to smooth $X \to \Spec W_m(k)$ admitting a lift of Frobenius \cite{berthelot12}. In this section, we apply these ideas to compute $\RGet(X_{inf},\mathbb G_m)$ and thus compute the groupoid of twistings of the infinitesimal site.

From now on $S = \Spec W_m(k)$ for $1 \leq m < \infty$.

\subsection{Stratifications via a lift of Frobenius}

Assume for this subsection that $X \to S$ is smooth. Let $X^{(r)}$ be the $r$-fold Frobenius twist of $X$, defined to be the pullback of $X$ along the Frobenius automorphism of $W_m(k)$.
We further assume that $X$ has a lift $\varphi: X \to X^{(r)}$ of the relative Frobenius.

The method suggested in \cite[(4.12) Remark]{ogus75} for calculating $\RGet(X_{inf},-)$ is that the infinitesimal groupoid $X \times_{X_{inf}} X$ is the colimit of the \v{C}ech nerves of $X\to X^{(r)}$, in other words, that as $\infty$-prestacks, $X_{inf} = \colim_r X^{(r)}$. 
To treat the entire infinitesimal groupoid $X \times_{ X_{inf}} X$ and not just the formal neighborhood of the diagonal in $X \times X$, we must generalize Berthelot's calculations in \cite[§1]{berthelot12} slightly. 

We begin by recalling the main lemma of \cite{berthelot12}. Let $A$ be a commutative ring with an ideal $I \subseteq A$.
Given such, define $I^{(i)} = (a^{p^i} : a \in I)$ and $\widetilde{I^{(i)}} = I^{(i)} + p I^{(i-1)} + p^2 I^{(i-2)} + + \cdots + p^{i-1} I + p^i A$.
\begin{lemma}\cite[Lemma 1.3]{berthelot12}
	\label{lemma: berthelot-pth-power}
	If $a \in \widetilde{I^{(i)}}$, then $a^p \in \widetilde{I^{(i+1)}}$.
\end{lemma}

\begin{lemma}\cite[Lemma 1.4]{berthelot12}
	\label{lemma: frobenius on p-twisted symbolic powers}
	Let $\mathcal I_{r,n}$ be the ideal of the closed immersion $X^{(r)} \to \left(X^{(r)}\right)^{\times n}$. Then 
	$(\varphi^{\times n})^i (\mathcal I_{r,n}) \subseteq \widetilde{\mathcal I^{(i)}_{r-i,n}}$.
\end{lemma}
\begin{proof}
	For $n> 2$, the ideal $\mathcal I_{r,n}$ is generated by the pullbacks of $\mathcal I_{r,2}$ along all projections $(X^{(r)})^{\times n}  \to X^{(r)} \times X^{(r)}$. Thus, it suffices to show the Lemma when $n=2$.
	The case $n=2$ is exactly \cite[Lemma 1.4]{berthelot12}, proved using Lemma \ref{lemma: berthelot-pth-power}.
\end{proof}

\begin{corollary} \label{corollary: cofinal-chains}
	Let $\mathcal I$ be the ideal of the closed immersion $X \to X^{\times n}$. Then the sequences of ideals $\{\mathcal I^N\}_{N\geq 1}$ and $\{(\varphi^{\times n})^r \mathcal I^{(r)})\}_{r\geq 1}$ are locally cofinal.
\end{corollary}
\begin{proof}
	Since $X \to S$ is locally of finite type, the ideal $\mathcal I$ is locally finitely generated, and thus the powers $\{\mathcal I^N\}_{N \geq 1}$ and symbolic powers $\{\mathcal I^{(i)}\}_{i \geq 1}$ are locally cofinal. Since $p^m = 0$ on $S$, Lemma \ref{lemma: frobenius on p-twisted symbolic powers} implies
	\[\varphi^{\times r}(\mathcal I_r) \subseteq \widetilde{\mathcal I^{(r)}} \subseteq \mathcal I^{(r-m)}
	\] 
	for $r \geq m$. Conversely, if $a \in I$, then $a^p \equiv \varphi(a) \mod p$.
	It is well-known that for $x,y$ elements of a commutative ring and $p$ a prime number, if $x\equiv y \mod p^i$ then $x^p \equiv y^p \mod p^{i+1}$.
	Thus $a^{p^r} \equiv (\varphi^{r-m+1}(a))^{p^{m-1}}  = \varphi^{r-m+1}(a^{p^{m-1}})\mod p^m$, so 
	\[
		\mathcal I^{(r)} \subseteq ((\varphi^{\times n})^{r-m+1} \mathcal I^{(r-m+1)})
	\] 
	Thus the Frobenius pullbacks $\{(\varphi^{\times n})^r\mathcal I^{(r)}\}_{r \geq 1}$ and symbolic powers $\{\mathcal I^{(i)}\}_{i \geq 1}$ of $\mathcal I$ are cofinal.
\end{proof}

\begin{proposition} \label{proposition: derived-inverse-limit}
	If $X\to S = \Spec W_m(k)$ is a smooth morphism of finite type with a lift of Frobenius $\varphi$, then 
	\[ \RGet(X_{inf},\mathbb G_m) \simeq R\varprojlim_\varphi \RGet(X^{(r)},\mathbb G_m).\]
\end{proposition}
\begin{proof}
	Let $(X/X^{(r)})_\bullet$ be the \v{C}ech nerve of $\varphi^r: X \to X^{(r)}$.
	By Corollary \ref{corollary: cofinal-chains},
	$(X/X_{inf})_\bullet$ is the degreewise colimit of $(X/X^{(r)})_\bullet$.
	Since cohomology commutes with inverse limts, we have $\RGet(X_{inf},\mathbb G_m) = R\varprojlim \RGet((X/X^{(r)})_\bullet, \mathbb G_m)$.
	The functor $\RGet(-,\mathbb G_m)$ is an fppf sheaf by Grothendieck's version of Hilbert 90 \cite[Theorem 11.7]{GGK68}, so 
	\[ R\varprojlim \RGet((X/X^{(r)})_\bullet, \mathbb G_m) \simeq R\varprojlim_\varphi \RGet(X^{(r)},\mathbb G_m).\qedhere\]
\end{proof}

\begin{remark}
	This method provides a proof of Ogus' result that for smooth and proper $X /k$, the cohomology of the infinitesimal site with respect to the structure sheaf is $H^\bullet(X_{inf}, \OO) \cong H^\bullet(X,\OO)^s$, the subspace of $H^\bullet(X,\OO)$ on which the Frobenius acts as an isomorphism \cite{ogus75}. Indeed, more generally we obtain 
	\[ R\Gamma(X_{inf},\OO_{X_{inf}}) \simeq R\varprojlim_\varphi R\Gamma(X, \OO_X).\]
	By Lemma \ref{lemma: tate-module-sequence},
	this reduces to Ogus' formula if the cohomology groups $H^i(X,\OO_X)$ are finite-dimensional.	
\end{remark}

We recall the derived completion of an object $A$ with respect to an endomorphism $a$ is the homotopy inverse limit 
\[ A^\wedge_a = R\varprojlim_n \cone(\begin{tikzcd}
	A & A
	\arrow["{a^n}", from=1-1, to=1-2]
\end{tikzcd}),\]
and thus $A^\wedge_a$ is the cone of $(R\varprojlim_a A) \to A$; see Appendix \ref{appendix: completion}.

\begin{corollary}\label{corollary: phi-adic-completion}
	If $X\to S = \Spec W_n(k)$ is a smooth morphism of finite type with a lift of Frobenius $\varphi$, then 
	\[ \fib(\RGet(X_{inf},\mathbb G_m) \to \RGet(X,\mathbb G_m)) \simeq \RGet(X,\mathbb G_m)^\wedge_\varphi[-1].\]
\end{corollary}
\begin{proof}
	Follows from Proposition \ref{proposition: derived-inverse-limit} and definitions.
\end{proof}

\subsection{Syntomic cohomology in prime characteristic}

\begin{definition} \label{defn: syntomic-cohomology}
	Let $X$ be a scheme. The syntomic cohomology of $X$ in weight one is defined by
	\[ R\Gamma(X,\mathbb Z_p(1)) = \RGet(X,\mathbb G_m)^{\wedge}_p[-1].\]
\end{definition}

In characteristic $p$, pullback by Frobenius on $\mathbb G_m$ is multiplication by $p$. Thus $\varphi$-adic completion is $p$-adic completion:
\begin{corollary}\label{corollary: char p classification}
	If $X \to \Spec k$ is a smooth variety, then the groupoid of TDO's on $X$ is controlled by 
	\[ \tau_{\leq 2} R\Gamma(X,\mathbb Z_p(1)).\]
\end{corollary}
\begin{proof}
	Follows from Corollaries \ref{corollary: twistings-classified-by-fiber}, \ref{corollary: phi-adic-completion}, and Definition \ref{defn: syntomic-cohomology}.
\end{proof}

Suppose the Néron-Severi group $\NS(X) = \Pic(X)/\Pic^0(X)$ is a finitely generated abelian group, which holds if $X \to \Spec k$ is proper by the Theorem of the Base \cite{SPEC-theorem-of-the-base}. 
As $\Pic^0(X)$ is $p$-divisible, the $p$-adic completion of $\Pic(X)$ is identified with $\NS(X) \otimes \mathbb Z_p$.
Lemma \ref{lemma: tate-module-sequence} gives a short exact sequence 
\begin{equation*} 
	0 \to \NS(X)\otimes \mathbb Z_p \to H^2(X,\mathbb Z_p(1)) \to T_p\Br(X) \to 0,
\end{equation*}
which appears in \cite[(5.8.5)]{illusie79}.

If $k$ is algebraically closed, Illusie also showed that $R\Gamma(X,\mathbb Z_p(1))$ may be computed using the crystalline cohomology of $X$.
Let $W \Omega_X$ be the de Rham-Witt complex of $X/k$ and $W \Omega_X^{\geq 1} \subseteq W\Omega_X$ be the stupid truncation above degree 1. Let $F': W \Omega^{\geq 1} \to W\Omega^{\geq 1}$ be the crystalline Frobenius divided by $p$. Then there is a long exact sequence 
\[\begin{tikzcd}[cramped,column sep=small]
	\cdots & {H^\bullet(X, \mathbb Z_p(1))} & {H^\bullet(X, W\Omega^{\geq 1}_X)} & {H^\bullet(X, W\Omega^{\geq 1}_X)} & \cdots
	\arrow[from=1-1, to=1-2]
	\arrow[from=1-2, to=1-3]
	\arrow["{1 - F'}", from=1-3, to=1-4]
	\arrow[from=1-4, to=1-5]
\end{tikzcd}\]
\cite[p.\ 627, (5.5.2)]{illusie79}. 
Using that $k$ is closed under Artin-Schreier extensions, for $i = 1,2$ the above sequence splits into the short exact sequences \cite[p.\ 629, §5.8]{illusie79}
\[\begin{tikzcd}[cramped,column sep=small]
	0 & {H^i(X, \mathbb Z_p(1))} & {H^i(X, W\Omega^{\geq 1}_X)} & {H^i(X, W\Omega^{\geq 1}_X)} & 0
	\arrow[from=1-1, to=1-2]
	\arrow[from=1-2, to=1-3]
	\arrow["{1 - F'}", from=1-3, to=1-4]
	\arrow[from=1-4, to=1-5].
\end{tikzcd}\]

\begin{example}
	Suppose $k$ is algebraically closed and $X/k$ is an ordinary K3 surface.
	Then 
	$H^2(X,\mathbb Z_p(1)) \cong \mathbb Z_p^{20}$,
	so $T_p\Br(X) \cong \mathbb Z_p^{20 -\rho}$
	where $\rho = \rank \NS(X)$ is the Picard rank of $X$ \cite[p.\ 653, §7.2]{illusie79}.
	The Picard rank is often much less than 20; for example, the theorem of M.\ Noether implies that the Picard rank of a very general fourfold in $\mathbb P^3$ is $1$ \cite[Exposé XIX]{sga-7-2}. Note that $X$ being very general requires such $X$ to be defined over a transcendental extension of $\overline{\mathbb F_p}$, as the Picard rank of a K3 over $\overline{\mathbb F_p}$ is even by the Tate conjecture.
\end{example}	

\begin{example}[Zariski-locally trivial tdo's]
	A tdo $D$ is \emph{Zariski-locally trivial} if there is a Zariski cover $U_\alpha \subset X$ such that $D|_{U_\alpha} \cong \Diff_{U_\alpha}$. By \cite[Corollary IV.2.6]{milne-etale}, if $X \to \Spec k$ is smooth, then $\Br(X) \to \Br(K(X))$ is injective, 
	so the image of $[D]$ in $T_p\Br(X)$ must be zero.
	Thus, $D$ lies in the image of $\Pic(X)^\wedge_p$. Again assuming that $\NS(X)$ is finitely generated, every element of $\Pic(X)^\wedge_p \cong \NS(X) \otimes \mathbb Z_p$ is Zariski-locally trivial, as we may choose a cover of $X$ which simultaneously trivializes a finite generating set for $\NS(X)$.

	For example, if $k$ is algebraically closed, Zariski-locally trivial tdo's on $\mathbb P^n$ are classified by $\NS(\mathbb P^n) \otimes \mathbb Z_p \cong \mathbb Z_p$, which was found by D.P.\ Faurot using explicit Čech cocycles \cite{faurot94}. As $T_p\Br(\mathbb P^n) = T_p\Br(k) = 0$, all tdo's on $\mathbb P^n$ are Zariski-locally trivial.
\end{example}

\subsection{Comparison with syntomic cohomology in mixed characteristic}

Consider now a smooth morphism $X \to \Spec W_m(k)$ of finite type,
and let $X_0$ be the reduction of $X$ modulo $p$.
If $X$ admits a lift of Frobenius, then Corollary \ref{corollary: phi-adic-completion} describes the groupoid of tdo's on $X$.
In mixed characteristic, the Frobenius on $\mathbb G_m$ no longer agrees with multiplication by $p$, but they may be compared.
We show in Theorem \ref{theorem: mixed-characteristic-syntomic-comparison} that there is a morphism from $\tau_{\leq 2}R\Gamma(X,\mathbb Z_p(1))$ to the complex controlling tdo's on $X$, which may or may not be an isomorphism, depending on action of Frobenius on $R\Gamma(X_0,\OO_{X_0})$.
\begin{theorem}\label{theorem: mixed-characteristic-syntomic-comparison}
	Suppose that $X \to \Spec W_m(k)$ is a smooth morphism of finite type where $p = \ch k > 2$.
	\begin{enumerate}[(i)]
		\item The mapping fiber
	\[ \fib(\RGet(X_{inf},\mathbb G_m) \to \RGet(X,\mathbb G_m))\] 
	is derived $p$-complete, so that there is an essentially unique morphism
	\[
	R\Gamma(X,\mathbb Z_p(1)) \to \fib(\RGet(X_{inf},\mathbb G_m) \to \RGet(X,\mathbb G_m))
	\]
	fitting into a factorization
	\[\begin{tikzcd}[cramped]
		{\RGet(X,\mathbb G_m)[-1]} \\
		{R\Gamma(X,\mathbb Z_p(1))} & {\fib(\RGet(X_{inf},\mathbb G_m) \to \RGet(X,\mathbb G_m))}
		\arrow[from=1-1, to=2-2]
		\arrow[from=1-1, to=2-1]
		\arrow[dashed, from=2-1, to=2-2]
	\end{tikzcd}.\]
	\item Suppose that pullback by the relative Frobenius $\varphi$ of $X_0$ induces an isomorphism 
	\[ \varphi^*: H^i(X_0,\OO_{X_0}) \to H^i(X_0,\OO_{X_0})\]
	for $0 \leq i \leq 3$. Then reduction modulo $p$ defines an equivalence
	\[\begin{tikzcd}[cramped]
		{\tau_{\leq 2} \fib(\RGet(X_{inf},\mathbb G_m) \to \RGet(X,\mathbb G_m))} \\
		{\tau_{\leq 2} \fib(\RGet((X_0)_{inf},\mathbb G_m) \to \RGet(X_0,\mathbb G_m))}
		\arrow["\sim"', from=1-1, to=2-1]
	\end{tikzcd},\]
	so that tdo's on $X$ are classified by $\tau_{\leq 2} R\Gamma(X_0,\mathbb Z_p(1))$.

	\item Suppose that pullback by the relative Frobenius $\varphi$ of $X_0$ is zero on $H^i(X_0,\OO_{X_0})$ for $1 \leq i \leq 3$. Then the morphism from (i) induces an isomorphism on $H^1$ and $H^2$; in particualr, isomorphism classes of tdo's on $X$ are in bijection with $H^2(X, \mathbb Z_p(1))$.
	\end{enumerate}
\end{theorem}
\begin{proof}
	Let $Z$ be any scheme over $W_m(k)$, and let $Z_1$ be its reduction modulo $p$. Then there is a short exact sequence of sheaves 
	\begin{equation}\label{eq: exponential-triangle}
		\begin{tikzcd}
		0 & {p\OO_Z} & {\OO_Z^\times} & {\OO_{Z_1}^\times} & 0
		\arrow[from=1-1, to=1-2]
		\arrow["\exp", from=1-2, to=1-3]
		\arrow[from=1-3, to=1-4]
		\arrow[from=1-4, to=1-5]
		\end{tikzcd},
	\end{equation}
	where $\exp: p\OO_Z \cong 1 + p\OO_Z$ makes sense because the $p$-adic exponential has radius of convergence $p^{-1/(p-1)} > p^{-1}$ for $p > 2$.
	Applying this short exact sequence to the \v{C}ech nerve $X_\bullet$ of $X \to X_{inf}$ and the corresponding reduction $(X_0)_\bullet$ of $X_0 \to (X_0)_{inf}$ gives a diagram 
	\begin{equation}\label{eq: exponential-inf-diagram}
	\begin{tikzcd}[cramped, column sep = scriptsize]
	{R\Gamma(X_{inf},p\OO_{X_{inf}})} & {\RGet(X_{inf},\mathbb G_m)} & {\RGet((X_0)_{inf},\mathbb G_m)} & {} \\
	{R\Gamma(X,p\OO_X)} & {\RGet(X,\mathbb G_m)} & {\RGet(X_0,\mathbb G_m)} & {}
	\arrow[from=1-1, to=1-2]
	\arrow[from=1-1, to=2-1]
	\arrow[from=1-2, to=2-2]
	\arrow[from=2-1, to=2-2]
	\arrow[from=2-2, to=2-3]
	\arrow[from=1-2, to=1-3]
	\arrow[from=1-3, to=2-3]
	\arrow["{+1}", from=1-3, to=1-4]
	\arrow["{+1}", from=2-3, to=2-4]
	\end{tikzcd}
	\end{equation}
	of distinguished triangles. Now $p^{m-1}$ acts as zero on the left column, and derived completion is triangulated, so the fiber of the left column is derived $p$-complete. 
	By Corollary \ref{corollary: phi-adic-completion}, the fiber of $\RGet((X_0)_{inf},\mathbb G_m) \to \RGet(X_0, \mathbb G_m)$ is derived $p$-complete .
	Again since derived completion is triangulated, the fiber of the middle column is derived $p$-complete. The universal property of derived completion (Corollary \ref{corollary: completion-left-adjoint}) gives the desired factorization through $R\Gamma(X, \mathbb Z_p(1))$, completing proof of (i).

	Now suppose we are in the situation of (ii).
	Taking $\tau_{\leq 2}$ of \eqref{eq: exponential-inf-diagram}, it suffices to show that 
	\[ 
		\tau_{\leq 2} \left(R\Gamma(X_{inf},p\OO_{X_{inf}})\to R\Gamma(X,p\OO_X)\right)
	\]
	is an equivalence.
	By taking the $p$-adic filtration, it suffices to show 
	\[ \tau_{\leq 2}\left(R\Gamma((X_0)_{inf},\OO_{(X_0)_{inf}}) \to R\Gamma(X_0, \OO_{X_0})\right)\]
	is an equivalence, that is, $\tau_{\leq 2} (R\Gamma(X_0,\OO_{X_0})^\wedge_\varphi ) \simeq 0$.
	If $\varphi^*$ is an isomorphism on $H^i(X_0,\OO_{X_0})$ for $0 \leq i \leq 3$, then Lemma \ref{lemma: tate-module-sequence} shows $H^i(R\Gamma(X_0,\OO_{X_0})^\wedge_\varphi) = 0$ for $0 \leq i \leq 2$, as desired.

	Now suppose we are in the situation of (iii). 
	Taking $p$-adic completion of \eqref{eq: exponential-triangle} where $Z = X$ gives a triangle 
	\[ R\Gamma(X, p\OO_X)^\wedge_p[-1] \to R\Gamma(X,\mathbb Z_p(1)) \to R\Gamma(X_0,\mathbb Z_p(1)) \to^{+1}.\]
	Comparing this triangle to \eqref{eq: exponential-inf-diagram} and taking the $p$-adic filtration on $p\OO_X$ implies that we must show 
	\[H^i(R\Gamma(X_0,\OO_{X_0})^\wedge_p) \to H^i(R\Gamma(X_0, \OO_{X_0})^\wedge_\varphi)\]
	is an isomorphism for $i \in \{1,2\}$.
	The $k$-vector space $R\Gamma(X_0,\OO_{X_0})$ is $p$-complete, while the hypothesis that $\varphi^*: H^i(X_0,\OO_{X_0}) \to H^i(X_0,\OO_{X_0})$ is zero for $1 \leq i \leq 3$ implies 
	\[ R^i\Gamma(X_0,\OO_{X_0}) \to H^i(R\Gamma(X_0,\OO_{X_0})^\wedge_\varphi)\]
	is an isomorphism for $i\in \{1,2\}$ by Lemma \ref{lemma: tate-module-sequence}.
\end{proof}

\begin{remark}
	If one only cares about isomorphism classes of tdo's, then only $H^2(X_0,\OO_{X_0})$ and $H^3(X_0,\OO_{X_0})$ require control.
\end{remark}

\begin{remark}
	If $X \to \Spec W_m(k)$ is proper, then the hypothesis on $H^3(X_0,\OO_{X_0})$ in Theorem \ref{theorem: mixed-characteristic-syntomic-comparison}, parts (ii) and (iii) above may be dropped.
	By Lemma \ref{lemma: tate-module-sequence}, 
	$H^3(X_0,\OO_{X_0})$ only appears in the calculation of $\tau_{\leq 2}(R\Gamma(X_0,\OO_{X_0})^\wedge_{\varphi})$ through the Tate module $T_\varphi H^3(X_0,\OO_{X_0})$.
	The Tate module of a semilinear endomorphism of a finite-dimensional vector space is zero, 
	so if $X\to\Spec W_m(k)$ is proper, $T_\varphi H^3(X_0,\OO_{X_0}) = 0$.
\end{remark}

\begin{example}
	Lemma \ref{lemma: tate-module-sequence} furnishes us with a map 
	\[ \Pic(X)^\wedge_p \to H^2(X,\mathbb Z_p(1)) \to \{\text{tdo's on }X\}/\cong.\]
	If $\mathcal L_r$ is a $p$-adic Cauchy sequence in $\Pic(X)$, then the associated tdo is constructed as follows:
	$\Diff_{\leq i}(\mathcal L_r)$ is independent of $r$ for $r$ large enough, as differential operators of order at most $i$ commute with $p^r$th powers for $r$ large enough.
	Setting $D_{\leq i}$ to be the stable value of $\Diff_{\leq i}(\mathcal L_r)$ defines the filtered pieces of a tdo.
\end{example}

\section{Explicit constructions of tdo's in prime characteristic}

Suppose $S = \Spec k$, the case of prime characteristic. 
In this section, we produce an explicit bijection between tdo's on $X/S$ up to isomorphism and $H^2(X,\mathbb Z_p(1))$, independent of the method of Gaitsgory-Rozenblyum discussed in §\ref{subsection: stacky approach}. We also compare twisted Grothendieck differential operators to twisted crystalline differential operators.

\subsection{Differential operators on torsors}
The ring of Grothendieck differential operators $\Diff_X$ in characteristic $p$ is a union of matrix algebras $\End_{\OO_{X^{(r)}}}(\OO_X)$ \cite{chase74, berthelot96}. When $\Diff_X$ is replaced by a tdo, we show instead that $\Diff_X$ is a union of Azumaya algebras.

\begin{definition}
	Given a tdo $D$ on $X$, define $D^r$ to be the centralizer of $\OO_{X^{(r)}}$ in $D$. 
\end{definition}
By Corollary \ref{corollary: cofinal-chains}, we have $D = \cup_{r \geq 1} D^r$.

\begin{remark}
	The subalgebra $\Diff_X^r = \End_{\OO_{X^{(r)}}}(\OO_X)\subseteq \Diff_X$ is the image of Berthelot's differential operators of level $r-1$ \cite{berthelot96}.
\end{remark}

\begin{proposition}\label{prop: centralizer-is-azumaya}
	$D^r$ is an Azumaya algebra over $X^{(r)}$ which splits along the relative Frobenius $X \to X^{(r)}$.
\end{proposition}
\begin{proof}
	The claim is local. If $X$ has étale coordinates $t_1,\ldots, t_n$ dual to derivations $\partial_1,\ldots, \partial_n$, then $\gr(D^r)$ has a basis of those principal symbols $\partial_1^{(i_1)} \cdots \partial_n^{(i_n)}$ such that $i_j < p^r$ for all $j$. Hence $D^r$ is a locally free $\OO_{X^{(r)}}$ module of rank $p^{2r \dim X}$.
	Now the map $\OO_X \otimes_{\OO_{X^{(r)}}} D^r \to \End_{\OO_X}(D^r)$ by left and right multiplication may be checked to be an isomorphism by passing to associated graded.
\end{proof}

Hence $D^r$ defines an Azumaya algebra on $X^{(r)}$ which contains $\OO_X$ as a maximal commutative subalgebra.
As $k$ is perfect, we may identify $X^{(r)}$ with $X$ as a scheme; then $D^r$ is an Azumaya algebra with a canonical splitting of its pullback by the $r$th power of the Frobenius, and thus defines a $\mu_{p^r}$-gerbe on $X_{fppf}$. 
We now explicitly construct this gerbe using differential geometry.

\begin{lemma}[\cite{illusie79},(5.1.3)]\label{lemma: fppf-to-étale}
	If $X \to \Spec k$ is a scheme over a perfect field $k$ and $\varepsilon: X_{fppf} \to X_{\acute{e}t}$ is the canonical projection,
	then 
	\[R\varepsilon_\ast \mu_{p^r} \cong \OO_X^\times/(\OO_X^\times)^{p^r}[-1].\]
\end{lemma}
\begin{proof}
	Consider the short exact sequence
	\[\begin{tikzcd}
	0 & {\mu_{p^r}} & {\mathbb G_m} & {\mathbb G_m} & 0
	\arrow[from=1-1, to=1-2]
	\arrow[from=1-2, to=1-3]
	\arrow[from=1-3, to=1-4]
	\arrow[from=1-4, to=1-5]
	\end{tikzcd}\]
	on $X_{fppf}$.
	By Hilbert's Theorem 90, $R\varepsilon_\ast \mathbb G_m = \mathbb G_m$, so comparing the triangle $R\varepsilon_\ast \mu_{p^r} \to \mathbb G_m \to \mathbb G_m \to^{+1}$ to the short exact sequence 
	\[\begin{tikzcd}
	0 & {\OO_X^\times} & {\OO_X^\times} & {\OO_X^\times/(\OO_X^\times)^{p^r}} & 0
	\arrow[from=1-1, to=1-2]
	\arrow["{p^r}", from=1-2, to=1-3]
	\arrow[from=1-3, to=1-4]
	\arrow[from=1-4, to=1-5]
	\end{tikzcd}\]
	on $X_{\acute{e}t}$ gives the result.
\end{proof}

Thus, a $\mu_{p^r}$-gerbe on $X_{fppf}$ is the same as an $\OO_X^{\times}/(\OO_X^\times)^{p^r}$-torsor on $X_{\acute{e}t}$.
If $\mathcal L$ is a torsor over $\OO_X^\times/(\OO_X^\times)^{p^r}$, then taking $\Diff_{\mathcal L}^r$ makes sense since $(\OO_X^\times)^{p^r}$ centralizes such differential operators.

\begin{proposition}\label{proposition: diff-ops-on-torsors}
	Consider the isomorphism 
	\[ H^2(X_{fppf},\mu_{p^r}) \cong H^1(X_{\acute{e}t},\OO_X^\times/(\OO_X^\times)^{p^r})\]
	induced by Lemma \ref{lemma: fppf-to-étale}
	If $[D^r]$ is mapped to the $\OO_X^\times/(\OO_X^\times)^{p^r}$-torsor $\mathcal L$,
	then $D^r \cong \Diff_{\mathcal L}^r$ as associative algebras over $\OO_X$.
\end{proposition}
\begin{proof}
	By Proposition \ref{prop: centralizer-is-azumaya}, the algebra $D^r$ is Azumaya and thus splits \'etale locally. 
	\'Etale covers of $X^{(r)}$ are of the form $\{U_\alpha^{(r)} \to X^{(r)}\}$ where $\{U_\alpha \to X\}$ is an étale cover.
	If $D^r$ splits on $U_\alpha^{(r)} \to X^{(r)}$, then 
	\[
		D^r|_{U_\alpha} \cong End_{\OO_{U_\alpha}^{(r)}}(\mathcal E_\alpha)
	\] 
	for some sheaf $\mathcal E$ on $X^{(r)}$. Using $\OO_X \to D^r$ gives $\mathcal E_\alpha$ the structure of a $Fr^r_\ast\OO_X$-module where $Fr^r: X \to X^{(r)}$ is the $r$-fold relative Frobenius. We claim that 
	\[ 
		\mathcal E_\alpha = Fr^r_\ast \mathcal L_\alpha
	\] 
	for some line bundle $\mathcal L_\alpha$.
	It suffices to check this locally on $X^{(r)}$, so 
	we may base change to the algebraic closure of $k$ and consider a closed point $p$ with local coordinates $t_1,\ldots, t_n$.
	We now replace $X$ with $(Fr^r)^{-1}(x)$.
	If $t_1,\ldots, t_n$ are local coordinates at $p$, then the coordinate ring of $(Fr^r)^{-1}(p)$ is $k[t_1,\ldots, t_n]/(t_1^{p^r},\ldots, t_n^{p^r})$. 
	The sheaf $\mathcal E$ is a faithful $D^r$-module and thus a faithful $Fr^r_*\OO_X$-module. 
	Hence $(t_1\cdots t_n)^{p^r-1}\mathcal E \neq 0$.
	If $e$ is a section of $\mathcal E$ such that $(t_1\cdots t_n)^{p^r-1}e \neq 0$, then since $(t_1\cdots t_n)^{p^r-1}$ generates the minimal nonzero ideal of $k[t_1,\ldots,t_n]/(t_1^{p^r},\ldots, t_n^{p^r})$, we see $Fr^r_*\OO_X\cdot e$ is a rank 1 free module over $Fr^r_*\OO_X$. But $Fr^r_*\OO_X\cdot e$ has the same rank as $\mathcal E$ as an $\OO_{X^{(r)}}$-module, so $\mathcal E = Fr^r_*\OO_X \cdot e$ is locally free of rank 1 over $Fr^r_*\OO_X$.
	
	We have shown that on a splitting cover $U_\alpha^{(r)} \to X^{(r)}$, 
	\[D^r|_{U_\alpha} \cong \End_{\OO_{X^{(r)}}}(Fr^r_\ast \mathcal L_\alpha) \cong \Diff_{\mathcal L_\alpha}^r.\] On overlaps $U_{\alpha\beta}$ an isomorphism $\Diff_{\mathcal L_\alpha}^r|_{U_{\alpha\beta}} \cong \Diff_{\mathcal L_\beta}^r|_{U_{\alpha\beta}}$ implies that $\mathcal L_\alpha$ and $\mathcal L_\beta$ differ by a $p^r$th power on $U_{\alpha\beta}$; hence $\mathcal L_\alpha / (\OO_X^\times)^{(p^r)}$ glue to a $\OO_X^\times/(\OO_X^\times)^{p^r}$-torsor on $X$. Conversely,
	given a torsor over $\OO_X^\times/(\OO_X^\times)^{p^r}$, it may be locally lifted to line bundles $\mathcal L_\alpha$, and the algebras $\Diff_{\mathcal L_\alpha}^r$ descend to an Azumaya algebra $D^r$.
\end{proof}

It is known that $H^2(X,\mathbb Z_p(1)) \cong \varprojlim H^2(X_{fppf},\mu_{p^r})$, as the system\\ $\{H^\ast(X_{fppf},\mu_{p^r})\}_{r \geq 1}$ satisfies the Mittag-Leffler condition \cite[p. 627]{illusie79}.

\begin{corollary}
	The bijections
	\[ \varprojlim\limits_r H^1(X_{\acute{e}t},\OO_X^\times/(\OO_X^\times)^{p^r})) \leftrightarrow H^2(X,\mathbb Z_p(1)) \leftrightarrow \{\text{tdo's on }X/k\}/\cong  \]
	send a system $(\mathcal L_r)_{r \geq 1}$ of compatible $\OO_X^\times/(\OO_X^\times)^{p^r}$-torsors to 
	\[ D = \cup_r \Diff_{\mathcal L_r}^r.\]
\end{corollary}

\subsection{Compatibility with twisted crystalline differential operators}\label{subsection: compatibility}

While $\Diff_X$ has divided powers of partial derivatives and is thus not finitely generated, there is a different algebra associated to $X$ known as the \emph{crystalline differential operators} $\mathcal D^{crys}_X$. The algebra of crystalline differential operators $\mathcal D^{crys}_X$ is the enveloping algebra of the Lie algebroid $T_{X/S}$ over $\OO_X$. It has a PBW filtration with associated graded $\gr \mathcal D^{crys}_X = \Sym_{\OO_X} T_{X/S}$. Unlike $\Diff_X$, the algebra $\mathcal D^{crys}_X$ has a large center isomorphic to $\OO_{T^*X^{(1)}}$. The center is parameterized by $s: \OO_{T^*X^{(1)}} = \Sym_{\OO_{X^{(1)}}} T_{X^(1)} \to \mathcal D^{crys}_X$ sending a function $f \in \OO_{X^{(1)}}$ to $Fr(f)$ and a vector field $\partial$ to $\partial^p - \partial^{\pop}$, where $\partial^{\pop}$ is the $p$th power of $\partial$ as a vector field \cite[§1.3]{bmr}. The inclusion $\OO_X \oplus T_{X/S} \to \Diff_X$ induces a map $\mathcal D^{crys}_X \to \Diff_X$
whose kernel is generated by the ideal of the zero section in $T^*X^{(1)}$ and whose image is $\Diff_X^1$ \cite[(2.2.5)]{bmr}.

Associated to any tdo $D$ on $X$, we will construct a map analogous to $\mathcal D^{crys}_X \to \Diff_X$.
Instead of crystalline differential operators, the domain will be more general filtered quantizations of $T^*X$ with a Frobenius splitting in the sense of \cite{bk08}. 
These have an explicit description in terms of restricted Lie algebroids, which we now recall.

\begin{definition}
	A \emph{restricted Lie algebroid} on $X/S$ is a Lie algebroid $\tau: \mathcal A \to T_{X/S}$ equipped with an operation $-^\pop: \mathcal A \to \mathcal A$ such that 
	\begin{enumerate}[(i)]
		\item $-^\pop$ makes the Lie algebra $\mathcal A$ into a restricted Lie algebra such that the anchor map $\tau$ is a morphism of restricted Lie algebras;
		\item for $f$ a section of $\OO_X$ and $x$ a section of $\mathcal A$,
		\[ (fx)^\pop = f^px^\pop + \tau(fx)^{p-1}(f)x. \]
	\end{enumerate}
\end{definition}

\begin{definition}[\cite{mundinger22}, Definition 4.2]
	A \emph{restricted Atiyah algebra} on $X$ is a restricted Lie algebroid of the form 
	\[ 0 \to \OO_X \to \mathcal A \to T_{X/S} \to 0\]
	such that for all sections $f$ of $\OO_X \subset \mathcal A$, we have $f^\pop = f^p$ and $[x,f] = \tau(x)f$ for $x \in \mathcal A$.
\end{definition}

\begin{proposition}[\cite{devalapurkar21}, Talk 11, Remark 23]
	The functor sending a restricted Atiyah algebra $\mathcal A$ to its enveloping algebra $U\mathcal A$ defines an equivalence between 
	restricted Atiyah algebras on $X$ and Frobenius-split filtered quantizations of $T^*X$.
\end{proposition}
\begin{proof}
	Filtered quantizations of $T^*X$ are of the form $U\mathcal A$ where $\mathcal A$ is an Atiyah algebra. 
	The algebra $\mathcal A$ has a filtered Frobenius splitting in the sense of Bezrukavnikov-Kaledin \cite{bk08} if and only if there is $-^{\pop}: \mathcal A \to \mathcal A$ such that $s(\tau(x)) = x^p - x^\pop$ defines a Frobenius splitting $s: \Sym_{\OO_X} T_{X/S} \to U\mathcal A$.
	The axioms of a Frobenius splitting then hold if and only if $-^\pop$ makes $\mathcal A$ into a restricted Atiyah algebra. 
\end{proof}

An isomorphism of restricted Atiyah algebras is an isomorphism of Lie algebroids respecting the $p$th power map $x \mapsto x^\pop$. Classification of restricted Atiyah algebras up to isomorphism is as follows: there is the exact sequence 
\begin{equation}\label{eq: milne-sequence}
	\begin{tikzcd}
		0 & {(\OO_X^\times)^p} & {\OO_X^\times} & {\Omega^1_{X/S,cl}} & {\Omega^1_{X/S}} & 0
		\arrow[from=1-1, to=1-2]
		\arrow[from=1-2, to=1-3]
		\arrow["d\log", from=1-3, to=1-4]
		\arrow["{1 - C}", from=1-4, to=1-5]
		\arrow[from=1-5, to=1-6]
	\end{tikzcd}
\end{equation}
of sheaves on $X_{\acute{e}t}$,
where $\Omega^1_{X/S,cl}$ is the étale sheaf of closed 1-forms and $C$ is the Cartier operator \cite[Proposition 4.14]{milne-etale}. It was proved in \cite[Theorem 4.5]{mundinger22} that restricted Atiyah algebras are classified by 
\[ H^1(X, (\begin{tikzcd}
	{\Omega^1_{X/S,cl}} & {\Omega^1_{X/S}}
	\arrow["{id - C}", from=1-1, to=1-2]
\end{tikzcd})).\]
In light of \eqref{eq: milne-sequence} and Lemma \ref{lemma: fppf-to-étale}, we have 
\begin{equation}\label{eq: milne sequence consequence}
	(\begin{tikzcd}
	{\Omega^1_{X,cl}} & {\Omega^1_X}
	\arrow["{id - C}", from=1-1, to=1-2]
\end{tikzcd}) \simeq \OO_X^\times/(\OO_X^\times)^p \simeq R\varepsilon_\ast \mu_p [1],
\end{equation}
on $X_{\acute{e}t}$, where $\varepsilon: X_{fppf} \to X_{\acute{e}t}$ is the canonical projection.
Hence isomorphism classes of restricted Atiyah algebras on $X \to \Spec k$ are in bijection with $H^2(X_{fppf},\mu_p)$. This was also known by D.\ Bragg independently at the time of publication of \cite{mundinger22}.

Additionally associated to a restricted Atiyah algebra $\mathcal A$ is its \emph{restricted enveloping algebra} $u\mathcal A = U\mathcal A / (x^p - x^\pop \mid x \in \mathcal A)$.

\begin{lemma}
	Suppose that $X \to \Spec k$ and $\mathcal A$ is a restricted Atiyah algebra. Then:
	\begin{enumerate}[(i)]
		\item $u\mathcal A$ is an Azumaya algebra on $X^{(1)}$ which contains $\OO_X$ as a maximal commutative subalgebra;
		\item the class of the associated $\mu_p$-gerbe $[u\mathcal A] \in H^2(X_{fppf},\mu_p)$ agrees with the class of $[\mathcal A]$.
	\end{enumerate}
\end{lemma}
\begin{proof}
	Since $U\mathcal A$ is a Frobenius-split quantization of $T^*X$, the algebra $U\mathcal A$ is Azumaya over its center $\OO_{T^*X^{(1)}}$.
	Reducing modulo the ideal of the zero section of $T^*X^{(1)} \to X^{(1)}$ gives (i).

	Now since $u\mathcal A$ is an Azumaya algebra on $X^{(1)}$ with maximal commutative subalgebra $\OO_X$, it defines a $\mu_p$-gerbe $[u\mathcal A] \in H^2(X_{fppf},\mu_p)$.
	Now let $\mathcal L$ be an $\OO_X^\times / (\OO_X^\times)^p$-torsor where $[\mathcal L]$ maps to $[u\mathcal A]$ under the equivalence $\OO_X^\times/(\OO_X^\times)^p \simeq R\varepsilon_*\mu_p[1]$.
	\[ d\log[\mathcal L] \in H^1(X, \begin{tikzcd}
		{\Omega^1_{X,cl}} & {\Omega^1_X}
		\arrow["{id - C}", from=1-1, to=1-2]
	\end{tikzcd}),\] 
	corresponds to the restricted Atiyah algebra $\Diff_{\mathcal L,\leq 1}$. Because of the isomorphisms \eqref{eq: milne sequence consequence}, $\mathcal A$ is isomorphic to some $\Diff_{\mathcal L, \leq 1}$, so it suffices to show the claim 
	\[ u\Diff_{\mathcal L, \leq 1} \to \Diff_{\mathcal L}^1\] 
	induced by $\Diff_{\mathcal L,\leq 1}\to \Diff_{\mathcal L}^1$ is an isomorphism. Such an isomorphism can be checked étale locally; after pulling back along an étale cover trivializing $\mathcal L$, the map is exactly $\mathcal D^{crys}_{X,\leq 1} \to \Diff_{X,\leq 1}$, which is an isomorphism since both are $\OO_X \oplus T_X$. Thus (ii) follows.
\end{proof}

\begin{proposition}
	Let $D$ be a tdo on $X \to \Spec k$. Then $D_{\leq 1}$ is a restricted Atiyah algebra via the identifications $\gr^0 D \cong \OO_X$ and $\gr^1 D \cong T_{X/S}$. Further, the restricted enveloping algebra $u D_{\leq 1}$ is isomorphic to $D^1$.
\end{proposition}
\begin{proof}
	Suppose $x$ is a section of $D_{\leq 1}$.
	Then $ad(x^p) = ad(x)^p$, so for $f_0,f_1$ sections of $\OO_X$,
	\[ [[x^p,f_0],f_1] = [ad(x)^p(f_0),f_1] = 0\]
	since $[Q,f_0] \in \OO_X$. It follows that $D_{\leq 1}$ is a restricted Atiyah algebra with $p$-operation $x^\pop = x^p$.
	The natural map $UD_{\leq 1} \to D$ thus factors through $uD_{\leq 1}$. As $uD_{\leq 1}$ is generated in degree 1, the image commutes with $\OO_X^p$, and passing to associated graded shows the map $uD_{\leq 1} \to D^1$ is an isomorphism.
\end{proof}

We have constructed the desired map $\Dtilde^{crys} = UD_{\leq 1} \to D$ whose kernel is $(x^p - x^\pop \mid x \in D_{\leq 1})$, whose image is $D^1$, and which agrees with $\gr(\mathcal D^{crys}_X \to \Diff_X)$ on associated graded. 

\appendix

\section{Derived completions}\label{appendix: completion}

Fix a stable $\infty$-category $\mathcal C$ with countable limits. We define the derived completion of an object of $\mathcal C$ at an endomorphism and give a universal property for it. Our treatment is a dual version of Arinkin-Gaitsgory's cocompletion in triangulated categories \cite[§3.1]{ag15}.

Given $\mathcal C$, the category of endomorphisms in $\mathcal C$ is the $\infty$-category 
\[ \mathcal C^{end} = \Fun(B\mathbb N, \mathcal C),\]
where $\mathbb N$ is the free monoid on one generator $1 \in \mathbb N$ and $B\mathbb N$ is the nerve of $\mathbb N$.
We will denote an object of $\mathcal C^{end}$ by $(A,a)$ where $A \in \mathcal C$ and $a \in \End(A)$ is the image of $1 \in \mathbb N$, although an object of $\mathcal C^{end}$ also involves coherent choices of powers of $a$.
We also have the full subcategory $\mathcal C^{aut} \to \mathcal C^{end}$ of those $(A,a)$ where $a$ is an equivalence.

\begin{definition}
	Let $(A,a) \in \mathcal C^{end}$.
	Then the \emph{derived colocalization} of $(A,a)$
	is the inverse limit 
	\begin{equation}\label{eq: derived-colocalization}
		T(A,a) = \varprojlim \left(
		\begin{tikzcd}
			\cdots & {(A,a)} & {(A,a)} & {(A,a)}
			\arrow["a", from=1-1, to=1-2]
			\arrow["a", from=1-2, to=1-3]
			\arrow["a", from=1-3, to=1-4]
		\end{tikzcd}
	\right)
	\end{equation}
	in $\mathcal C^{end}$.
\end{definition}
The endomorphism of $T(A,a)$ is induced by $a: (A,a) \to (A,a)$.

\begin{proposition}
	Derived colocalization is right adjoint to the inclusion $\mathcal C^{aut} \to \mathcal C^{end}$.
\end{proposition}
\begin{proof}
	Let $p: T(A,a) \to (A,a)$ be the projection onto the first factor in the inverse limit \eqref{eq: derived-colocalization}. It suffices to show that if $(B,b) \in \mathcal C^{aut}$, then 
	\[ p_*: \Hom((B,b),T(A,a)) \to \Hom((B,b),(A,a))\]
	is an equivalence. 
	By definition of limit,
	\begin{equation}\label{eq: universal property of limit}
		\Hom((B,b),T(A,a)) \simeq \varprojlim_a \Hom((B,b),(A,a)),
	\end{equation}
	but the action of $a$ on $\Hom((B,b),(A,a))$ is equivalent to the action of $b$ since morphisms in $\mathcal C^{end}$ are intertwiners. 
	The morphism $b$ is an equivalence, so $a_*: \Hom((B,b),(A,a)) \to \Hom((B,b),(A,a))$ is also an equivalence.
	Now a filtered limit along equivalences gives an equivalence
	\[ \varprojlim_a \Hom((B,b),(A,a)) \to \Hom((B,b),(A,a)),\]
	by e.g.\ \cite{mayponto}[Proposition 2.2.9].
\end{proof}

Thus the inclusion $\mathcal C^{aut} \to \mathcal C^{end}$ is a colocalization, so that we have a short exact sequence of categories 
\[\begin{tikzcd}
	{\mathcal C^{aut}} & {\mathcal C^{end}} & {\mathcal C^{end, c}}
	\arrow[shift left, from=1-1, to=1-2]
	\arrow[shift left, from=1-2, to=1-1]
	\arrow[shift left, from=1-2, to=1-3]
	\arrow[shift left, from=1-3, to=1-2]
\end{tikzcd}\]
where $(\mathcal C^{end})^\perp = \mathcal C^{end,c}$ will be the \emph{complete objects} in $\mathcal C^{end}$:

\begin{definition}\label{definition: complete}
	Let $(A,a) \in \mathcal C^{end}$.
	The \emph{completion} of $(A,a)$ is 
	\[ A^\wedge_a = \cone(T(A,a) \to (A,a)).\]
	$(A,a)$ is \emph{complete} if $T(A,a)$ is contractible, that is, $(A,a) \to A^\wedge_a$ is an equivalence.
\end{definition}

By definition, the full subcategory  $\mathcal C^{end,c}$ on complete objects of $\mathcal C^{end}$ is the right orthogonal of $\mathcal C^{aut}$,
which immediately implies:

\begin{corollary}\label{corollary: completion-left-adjoint}
	The inclusion $\mathcal C^{end,c} \to \mathcal C^{end}$ is right adjoint to completion $(A,a) \mapsto A^\wedge_a$.
\end{corollary}

Consider now the case when $\mathcal C = D(R)$ for a ring $R$.
On abelian groups, $R\lim$ over an $\mathbb N$-indexed diagram has cohomological dimension one \cite[Corollary 3.5.4]{weibel}.
Hence in $D(R)$ the cohomology of $A^\wedge_a$ may be expressed in terms of short exact sequences involving $H^*(A)$ and the action of $H^*(a)$.
If $M$ is a classical $\mathbb Z[t]$-module, recall the classical \emph{Tate module} $T_tM = \lim_r B[t^r]$, where the transition maps are given by multiplication by $t$.

\begin{lemma}
	\label{lemma: tate-module-sequence}
	Let $R$ be a ring and $(A,a) \in D(R)^{end}$.
	For all $i \in \mathbb Z$, there are short exact sequences 
	\[\begin{tikzcd}
		0 & {H^0(H^i(A)^\wedge_a)} & {H^i(A^\wedge_a)} & {T_aH^{i+1}(A)} & 0
		\arrow[from=1-1, to=1-2]
		\arrow[from=1-2, to=1-3]
		\arrow[from=1-3, to=1-4]
		\arrow[from=1-4, to=1-5]
	\end{tikzcd}\]
\end{lemma}
\begin{proof}
	The category $D(R)^{end}$ is equivalent to $D(R[a])$. 
	Now \cite[15.93.5, \href{https://stacks.math.columbia.edu/tag/0BKF}{Tag 0BKF}]{stacks-project} 
	gives the above short exact sequence for complexes of $R[a]$-modules completed at $a$.
\end{proof}


\printbibliography

\end{document}